\documentclass[a4paper,11pt]{article}
\input{xypic}
\usepackage{relsize}
\usepackage{float,caption}
\usepackage{hyperref,bookmark,booktabs}
\usepackage{amsmath, amsthm, amssymb}
\usepackage[top=1in, bottom=1in, left=1in, right=1in]{geometry}
\usepackage{graphicx} 
\usepackage{epstopdf} 
\usepackage{mdframed}
\usepackage{ragged2e}
\usepackage{mathrsfs}
\usepackage[shortlabels]{enumitem}
\usepackage{tikz-cd}
\usepackage{setspace}
\usepackage{stmaryrd}

\newtheorem{theorem}{Theorem}[section]
\newtheorem{corollary}[theorem]{Corollary}
\newtheorem{lemma}[theorem]{Lemma}

\newtheorem{definition}[theorem]{Definition}
\newtheorem{remark}[theorem]{Remark}
\newtheorem{example}[theorem]{Example}

\theoremstyle{definition}

\def\P{\mathcal{P}}
\def\Q{\mathcal{Q}}
\def\T{\mathfrak{T}}

\def\F{\mathcal{F}}

\usepackage[affil-it]{authblk}
\begin{document}
\setcounter{page}{1}
\title{
Novel fuzzy topologies from old through fuzzy primals
}
\author[1]{Zanyar A. Ameen\thanks{corresponding author: zanyar@uod.ac}}
\author[2]{Ramadhan A. Mohammed}
\author[3]{Tareq M. Al-shami}
\author[4,5]{Baravan A. Asaad}

\affil[1]{\scriptsize Department of Mathematics, College of Science, University of Duhok, Duhok 42001, Iraq}
\affil[2]{Department of Mathematics, College of Basic Education, University of Duhok, Duhok 42001, Iraq}
\affil[3]{Department of Mathematics, Sana'a University, P.O.Box 1247 Sana'a, Yemen}
\affil[4]{Department of Mathematics, Faculty of Science, University of Zakho, Zakho-42002, Iraq}
\affil[5]{Department of Computer Science, College of Science, Cihan University-Duhok, Iraq}
\date{}
\maketitle
\vspace{-1cm}
\begin{abstract}
In this paper, we introduce a novel fuzzy structure named \textquotedblleft fuzzy primal\textquotedblright. We study the essential properties and discuss basic operations on it. A fuzzy operator $(\cdot)^\diamond$ on the family of all fuzzy sets is
introduced here by applying the q-neighborhood structure to a primal fuzzy topological space along with the {\L}ukasiewicz disjunction. We explore the main characterizations of $(\cdot)^\diamond$. Then, we define another fuzzy operator, symbolized by $Cl^\diamond$, with the utilization of $(\cdot)^\diamond$. These fuzzy operators are studied in order to deduce a new fuzzy topology from the original one. Such a new fuzzy topology is called primal fuzzy topology. The fundamental structure, particularly a fuzzy base that generates primal fuzzy topologies, as well as many relationships between different fuzzy primals and fuzzy topologies, are also analyzed. Lastly, the concept of compatibility between fuzzy primals and fuzzy topologies is introduced, and some equivalent conditions related to this are examined. It is shown that if a fuzzy primal is compatible with a fuzzy topology, then the fuzzy base that generates the primal fuzzy topology is itself a fuzzy topology.

\end{abstract}

\section{Introduction}\label{sec1}\
By Zadeh \cite{zadeh1996fuzzy}, the fundamental idea of a fuzzy set was first presented as an extension of a classical set. Chang later proposed the idea of fuzzy topology \cite{chang1968fuzzy}. Lowen \cite{lowen1976fuzzy} offered an alternate definition of fuzzy topology. The idea can be viewed as a natural generalization of the general topology. Since then, other authors in this field have examined and conducted various general topological characteristics in fuzzy settings. 

A new weak structure can be created by omitting a particular postulate (or part of a postulate) in fuzzy topology. For instance, by dismissing the finite intersection condition, Abd-El-Monsef and Ramadan \cite{abd1987fuzzy} appeared with the concept of fuzzy supra topology. By disregarding the union condition in fuzzy topology, Ameen et al. \cite{ameen2022novel} established the concept of infra fuzzy topology. A Lowen's version of  infra fuzzy topology was appeared in \cite{das2021some,ameen2023correction} under the name of fuzzy infi-topology. 

Other mathematical fuzzy structure like fuzzy ideals, fuzzy filters, and fuzzy grills are essential in developing fuzzy topology. The concept of a fuzzy ideal was introduced by Sarkar \cite{sarkar1997fuzzy} while studying the notion of ideal fuzzy topology. G{\"a}hler \cite{gahler1995general} used fuzzy filters for producing certain types of fuzzy topologies. The structure of fuzzy grills were defined by Azad \cite{azad1981fuzzy} for the purpose studying fuzzy approximation spaces. Fuzzy grills are generalization of (crisp) grills proposed by Choquet \cite{choquet1947theorie} in 1947. Mukherjee and Das \cite{mukherjee2010fuzzy} constructed a fuzzy topology which can be called "grill fuzzy topology. It is worth remarking that all these fuzzy topologies are finer than the original fuzzy topologies and are normal subclasses of the set of all fuzzy topologies.  

Recently, Acharjee et al. \cite{acharjee2022primal} defined the notion of primals as the dual of grills. Then they used primals to generate a new fuzzy topology called primal topology. Following their direction, Al-shami et al. \cite{primal} introduced soft primal and studied their basic properties. Then soft primals are applied to generate a new soft topology called primal soft topology. This proves to us constructing new topologies over a single universe via some structures is another fruitful area of study. Motivating by the latter statements, we introduce the concept of fuzzy primals and study their fundamental properties. By using the fuzzy primal structure, we build a new fuzzy topology from an old one and name it the primal fuzzy topology.

Here is the remaining content of the paper: Afterward the introduction, we review the definitions and outcomes needed to truly comprehend the material provided in this paper. After that, in Section 3, we establish the concept of fuzzy primals and demonstrate how fuzzy primals and fuzzy grills are related concepts. The fundamental operations on fuzzy primals are investigated. In Section 4, we propose the definition of a primal fuzzy topological space followed by a fuzzy topological operator $(\cdot)^\diamond$.  Then, we study the main properties of $(\cdot)^\diamond$. Furthermore, we define another fuzzy operator called $Cl^\diamond$ with the aid of $(\cdot)^\diamond$ and show that $Cl^\diamond$ is the Kuratowski fuzzy closure operator. It means that $Cl^\diamond$ generates a fuzzy topology, which we call a primal fuzzy topology. The central properties of primal fuzzy topologies are studied. In Section 5, we introduce the concept of compatibility of a fuzzy primal with a fuzzy topology. Some equivalences of the this concept are investigated. In Section 6, we conclude by summarizing the major contributions and making some recommendations for the future.

\section{Preliminaries}\
Let $Y$ be a universe (domain), $\mathbb{I}$ be the unit interval $[0,1]$, and $\mathbb{I}^Y$ be the class of all fuzzy sets in $Y$. All undefined terminologies used in the manuscript can be found in \cite{chang1968fuzzy,pao1980fuzzy,zadeh1996fuzzy}.

\begin{definition}\cite{zadeh1996fuzzy}
A mapping $\mu$ from $Y$ to $\mathbb{I}$ is called a fuzzy set in $Y$. The value $\mu(y)$ is called the degree of the membership of $y$ in $\mu$ for each $y\in Y$. The support of $\mu$ is the set $\{y\in Y:\mu(y)>0\}$. The complement of $\mu$ is, denoted by $\bar{\mu}$ (or $1_Y-\mu$ if there is no confusion), given by $\bar{\mu}(y)=1-\mu(y)$ for all $y\in Y$. 
\end{definition}

\begin{definition}\cite{zadeh1996fuzzy}
Let $\mu, \nu\in\mathbb{I}^Y$. Then $\mu\subseteq\nu$ if $\mu(y)\leq\nu(y)$ for all $y\in Y$.
\end{definition}

\begin{definition}\cite{zadeh1996fuzzy}
Let $\{\mu_s:s\in S\}\subseteq \mathbb{I}^Y$, where $S$ is any index set. Then
\begin{enumerate}[(i)]
	\item $\bigcup\mu_s(y)=\sup\{\mu_s(y):s\in S\}$, for each $y\in Y$.
	\item $\bigcap\mu_s(y)=\inf\{\mu_s(y):s\in S\}$, for each $y\in Y$.
\end{enumerate}
\end{definition}

\begin{definition}\cite{lukasiewicz1970select}
Let $\mu,\nu\in\mathbb{I}^Y$. Then $(\mu+\nu)$ and $(\mu-\nu)$ are defined by
\[
(\mu+\nu)(y)=\begin{cases}
\mu(y)+\nu(y), &\text{ if } \mu(y)+\nu(y)\leq 1;\vspace{.3cm}\\
1, &\text{ if } \mu(y)+\nu(y)< 1,
\end{cases}
\]
and 
\[
(\mu-\nu)(y)=\begin{cases}
\mu(y)-\nu(y), &\text{ if } \mu(y)>\nu(y);\vspace{.3cm}\\
0, &\text{ if } \mu(y)\leq\nu(y),
\end{cases}
\]
for each $y\in Y$.
\end{definition}

\begin{definition}\cite{lukasiewicz1970select}
Let $\mu,\nu\in\mathbb{I}^Y$. Then,  for each $y\in Y$,
\begin{enumerate}[(i)]
\item $\mu\odot\nu=max(\mu(y)+\nu(y)-1,0)$. That is,
\[
(\mu\odot\nu)(y)=\begin{cases}
	\mu(y)+\nu(y)-1, &\text{ if } \mu(y)+\nu(y)>1;\vspace{.3cm}\\
	0, &\text{ if } \mu(y)+\nu(y)\leq 1,
\end{cases}
\]
\item $\mu\oplus\nu=min(\mu(y)+\nu(y),1)$.
 That is,
\[
(\mu\oplus\nu)(y)=\begin{cases}
	\mu(y)+\nu(y), &\text{ if } \mu(y)+\nu(y)<1;\vspace{.3cm}\\
	1, &\text{ if } \mu(y)+\nu(y)\geq 1.
\end{cases}
\]
\end{enumerate}
\end{definition}

Here, we give the following remark that will be used in the sequel: Its proof can be easily checked from the above definitions.
\begin{remark}\label{op-remark}
	Let $\mu,\nu,\eta,\theta\in\mathbb{I}^Y$. Then
	\begin{enumerate}[(i)]
		\item If $\mu\subseteq\nu$, then $\mu\oplus\eta\subseteq\nu\oplus\eta$.
		\item $(\mu\oplus\eta)\cap(\nu\oplus\eta)\subseteq(\mu\cap\nu)\oplus\eta$.
		\item $\overline{\mu\oplus\nu}=\bar{\mu}\oplus\bar{\nu}$.
	\end{enumerate}
\end{remark}

\begin{definition}\cite{pao1980fuzzy}
A fuzzy set $\mu\in\mathbb{I}^Y$ is said to be a fuzzy point, denoted by $y_t$, with the support $y\in Y$ and the value $t\in(0,1]$ if $y_t:Y\to\mathbb{I}$ is the mapping defined as follows: for each $z\in Y$,
	\[y_t(z)=
	\begin{cases}
		t  &\text{ if }z=y;\vspace{.3cm}\\
		0  &\text{ otherwise}.
	\end{cases}
	\]
	
A fuzzy point $y_t$ said to be in $\mu$, denoted by $y_t\in\mu$, if $t\leq\mu(y)$. We write $\F_t(Y)$ for the set of all fuzzy points in $Y$.
\end{definition}

\begin{definition}\cite{pao1980fuzzy}
Let $Y\ne\emptyset$, and let $\mu,\nu\in \mathbb{I}^Y$. Then $\mu,\nu$ intersect if there is $y\in Y$ such that $\mu(y)\wedge\nu(y)\neq 0$.
\end{definition}

\begin{definition}\cite{pao1980fuzzy}
Let $y_t\in \F_t(Y)$ such that $Y\ne\emptyset$ and let $\mu,\nu\in \mathbb{I}^Y$. Then
\begin{enumerate}[(i)]
		\item $y_t$ is called quasi-coincident with $\mu$ if $t+\mu(y)>1$ and is denoted $y_t\prec\mu$.
		\item  $y_t$ is called not quasi-coincident with $\mu$ if $t+\mu(y)\leq1$ and is denoted $y_t\nprec\mu$.
		\item  $\mu$ is called quasi-coincident with $\nu$ (at $y$) if there exists $y\in Y$ such that $\mu(y)+\nu(y)>1$ and is denoted $\mu\prec\nu$.
	\end{enumerate}
\end{definition}

\begin{lemma}\cite[Proposition 2.1]{pao1980fuzzy}\label{quasi}
Let $y_t\in \F_t(Y)$ such that $Y\ne\emptyset$ and $\mu,\nu\in \mathbb{I}^Y$. Then
\begin{enumerate}[(i)]
	\item if $\mu\prec\nu$ at $y$, then $\mu(y)\wedge\nu(y)\neq 0$.
	\item $y_t\in\mu$ iff $y_t\nprec 1-\mu$.
	\item $\mu\leq\nu$ iff $\mu\nprec 1-\nu$.
\end{enumerate}
\end{lemma}

\begin{definition}\label{top}\cite{chang1968fuzzy}
A subcollection $\T$ of $\mathbb{I}^Y$ is said to be a fuzzy topology on $Y$ if
\begin{enumerate}[(i)]
		\item $0_Y,1_Y\in \T$.
		\item  $\mu_1\cap\mu_2\in \T$ whenever $\mu_1,\mu_2\in \T$.
		\item  $\bigcup\mu_s(y)\in\T$ whenever $\{\mu_s(y):s\in S\}\subseteq\T$.
\end{enumerate}
The pair $(Y,\T)$ is called a fuzzy topological space. Members of $\T$ are called fuzzy $\T$-open (or simply open) subsets of $Y$ and their complements are called fuzzy $\T$-closed (or simply closed). The family of all fuzzy closed set is denoted by $\T^c$.
\end{definition}


\begin{definition}\cite{pao1980fuzzy}
Let $(Y,\T)$ be a fuzzy topological space and let $y_t\in\F_t(Y)$. A fuzzy subset $\eta$ of $(Y,\T)$ is called a fuzzy neighborhood of $y_t$ if there exists $\mu\in\T$ such that $y_t\in\mu\subseteq\eta$. The family of all fuzzy neighborhoods of $y_t$ is denoted by $\mathcal{N}_\T(y_t)$ or simply $\mathcal{N}(y_t)$.
\end{definition}

\begin{definition}\cite{pao1980fuzzy}
Let $(Y,\T)$ be a fuzzy topological space and let $y_t\in\F_t(Y)$. A fuzzy subset $\eta$ of $(Y,\T)$ is called a fuzzy q-neighborhood of $y_t$ if there exists $\mu\in\T$ such that $y_t\prec\mu\subseteq\eta$. The family of all fuzzy q-neighborhoods of $y_t$ is denoted by $\mathcal{Q}^*(y_t)$. If $\eta\in\T$, then it is called an open q-neighborhood of $y_t$. The family of all fuzzy open q-neighborhoods of $y_t$ in $\T$ is denoted by $\mathcal{Q}_\T(y_t)$ or simply $\mathcal{Q}(y_t)$.
\end{definition}

\begin{definition}\cite{pao1980fuzzy}\label{fuzzybase}
	Let $(Y,\T)$ be a fuzzy topological space. A subfamily $\mathcal{B}$ of $\T$ is called a fuzzy base for $\T$ if for each $y_t\in\F_t(Y)$ and each $\mu\in\Q(y_t)$, there exists $\beta\in\mathcal{B}$ such that $y_t\prec\beta\subseteq\mu$.
\end{definition}

\begin{definition}\cite{pao1980fuzzy}
Let $(Y,\T)$ be a fuzzy topological space and let $\lambda\in\mathbb{I}^Y$. A fuzzy point $y_t\in\F_t(Y)$ is called a fuzzy interior point of $\lambda$ if $y_t$ has a fuzzy neighborhood included in $\lambda$. The set of all fuzzy interior points of $\lambda$ is symbolized by $Int(\lambda)$
\end{definition}

\begin{definition}\cite{pao1980fuzzy}
Let $(Y,\T)$ be a fuzzy topological space and let $\lambda\in\mathbb{I}^Y$. A fuzzy point $y_t\in\F_t(Y)$ is called a fuzzy closure point of $\lambda$ if each fuzzy q-neighborhood of $y_t$ is quasi-coincident with $\lambda$.  The set of all fuzzy closure points of $\lambda$ is symbolized by $Cl(\lambda)$
\end{definition}

\begin{definition}\cite{lowen1976fuzzy}\label{Kuratowski}
	A fuzzy operator $c:\mathbb{I}^Y\to\mathbb{I}^Y$ is called a fuzzy (Kuratowski) closure operator if it satisfies the following axioms for each $\mu,\nu\in\mathbb{I}^Y$:
	\begin{enumerate}[(i)]
		\item $c(0_Y)=0_Y$.
		\item $\mu\subseteq c(\mu)$.
		\item $c(\mu\cup\nu)=c(\mu)\cup c(\nu)$.
		\item $c(c(\mu))=c(\mu)$.
	\end{enumerate}
\end{definition}

\begin{definition}\label{grilldefinition}\cite{azad1981fuzzy}
	A subfamily $\mathcal{G}$ of $\mathbb{I}^Y$ is said to be a fuzzy grill on $Y$ if it satisfies the following postulates:
	\begin{enumerate}[(i)]
		\item $0_Y\notin\mathcal{G}$.
		\item  If $\mu\in \mathcal{G}$, $\nu\in\mathbb{I}^Y$, and $\mu\subseteq\nu$, then $\nu\in \mathcal{G}$. 
		\item  If $\mu,\nu\in\mathbb{I}^Y$ such that $\mu\cup \nu\in \mathcal{G}$, then $\mu\in \mathcal{G}$ or $\nu\in \mathcal{G}$.
	\end{enumerate}
\end{definition}

\section{Fuzzy primal}\
\begin{definition}\label{primaldefinition}
	A subfamily $\mathcal{P}$ of $\mathbb{I}^Y$ is said to be a fuzzy primal on $Y$ if it satisfies the following postulates:
	\begin{enumerate}[(i)]
		\item $1_Y\notin\mathcal{P}$.
		\item If $\mu\in \mathcal{P}$, $\nu\in\mathbb{I}^Y$, and $\nu\subseteq\mu$, then $\nu\in \mathcal{P}$. 
		\item If $\mu,\nu\in\mathbb{I}^Y$ such that $\mu\cap \nu\in \mathcal{P}$, then $\mu\in \mathcal{P}$ or $\nu\in \mathcal{P}$.
	\end{enumerate}
\end{definition}

The next result is easy to prove.
\begin{theorem}\label{pp1}
	A subfamily $\mathcal{P}$ of $\mathbb{I}^Y\ $ is a fuzzy primal on $Y$ iff the following conditions are satisfied.
	\begin{enumerate}[(i)]
		\item $1_Y\not\in\mathcal{P}$.
		\item If $\mu\not\in\mathcal{P}$ and $\mu\subseteq\nu$, then $\nu\not\in\mathcal{P}$.
		\item If $\mu\not\in \mathcal{P}$ and $\nu\not\in\mathcal{P}$, then $\mu\cap \nu\not\in\mathcal{P}$.
	\end{enumerate}
\end{theorem}

\begin{example}\label{examp}
The following families are fuzzy primals:
\begin{enumerate}[(i)]
	\item $\P_1=\mathbb{I}^Y-\{1_Y\}$ (Trivial fuzzy primal).
	\item $\P_2=\{\lambda\in\mathbb{I}^Y:\lambda\leq\mu\}$, where 
	\[
	\mu(x)=\begin{cases}
		4x, &\text{ if } 0\leq x\leq 1/4;\vspace{.3cm}\\
		1, &\text{ if } 1/4\leq x\leq 1.
	\end{cases}
	\]
\item  $\P_3=\{\lambda\in\mathbb{I}^Y, y_{t_0}\notin\lambda\}$, where $y_{t_0}\in\F_t(Y)$.
\item  $\P_4=\{\lambda\in\mathbb{I}^Y, \lambda\text{ has a range which is subset of even numbers}\}$.
\end{enumerate}
\end{example}
It is worth noting that the families $\mathbb{I}^Y$ and $\{0_Y\}$ are not fuzzy primals. The first is clear as $1_Y\in\mathbb{I}^Y$. For the latter, we have $\mu\cap\nu=0_Y\in\{0_Y\}$, but neither $\mu$ nor $\nu$ belongs to $\{0_Y\}$, where 
	\[
\mu(x)=\begin{cases}
	0, &\text{ if } 0\leq x\leq 1/2;\vspace{.3cm}\\
	2x-1, &\text{ if } 1/2\leq x\leq 1,
\end{cases}
\]
and
	\[
\nu(x)=\begin{cases}
1-2x, &\text{ if } 0\leq x\leq 1/2;\vspace{.3cm}\\
0, &\text{ if } 1/2\leq x\leq 1.
\end{cases}
\]

\begin{theorem}
	If $\mathcal{G}$ is a fuzzy grill on $Y$, then the family $\mathcal{P}=\{\nu: \bar{\nu}\in \mathcal{G}\}$ is a fuzzy primal on $Y$.
\end{theorem}
\begin{proof}
	First, it is obvious that $0_Y\not\in \mathcal{G}$, so $1_Y\not\in \mathcal{P}$. Second, let $\nu\in \mathcal{P}$ and take any fuzzy subset $\mu$ of $\nu$. By the way of constructing $\mathcal{P}$ we have $\bar{\nu}\in \mathcal{G}$. Since $\bar{\nu}\subseteq\bar{\mu}$, it follows from the definition of fuzzy grill that $\bar{\mu}\in \mathcal{G}$. This automatically means that $\mu\in \mathcal{P}$. Third, let $\mu, \nu$ be fuzzy subsets such that $\mu\cap\nu\in \mathcal{P}$. Then, $\bar{\mu}\cup\bar{\nu}\in \mathcal{G}$. Therefore, $\bar{\mu}\in\mathcal{G}$ or $\bar{\nu}\in\mathcal{G}$. Thus, $\mu\in\mathcal{P}$ or $\nu\in\mathcal{P}$. Hence, we get the desired result.
\end{proof}

\begin{corollary}
	If $\mathcal{P}$ is a fuzzy primal on $Y$, then the family $\mathcal{G}=\{\nu: \bar{\nu}\in \mathcal{P}\}$ is a fuzzy grill on $Y$.
\end{corollary}

\begin{theorem}
	If $\mathcal{P}_1$ and $\mathcal{P}_2$ are two fuzzy primals on $Y$, then $\mathcal{P}_1\cup\mathcal{P}_2$ is a fuzzy primal on $Y$.
\end{theorem}
\begin{proof}
	First, let $\mathcal{P}_1$ and $\mathcal{P}_2$ be two fuzzy primals on $Y$. Then $1_Y\not\in \mathcal{P}_1$ and $1_Y\not\in \mathcal{P}_2$. So that $1_Y\not\in \mathcal{P}_1\cup\mathcal{P}_2$. Second, suppose that $\nu\in \mathcal{P}_1\cup\mathcal{P}_2$ and let $\mu\subseteq\nu$. Then, $\nu\in \mathcal{P}_1$ or $\nu\in\mathcal{P}_2$. This automatically leads to that $\mu\in \mathcal{P}_1$ or $\mu\in\mathcal{P}_2$. So $\mu\in \mathcal{P}_1\cup\mathcal{P}_2$. Third, let $\mu, \nu$ be fuzzy subsets such that $\mu\cap\nu\in \mathcal{P}_1\cup\mathcal{P}_2$. Then, $\mu\cap\nu\in\mathcal{P}_1$ or $\mu\cap\nu\in\mathcal{P}_2$. This implies that $\mu\in \mathcal{P}_1\cup\mathcal{P}_2$ or $\nu\in \mathcal{P}_1\cup\mathcal{P}_2$, as required.
\end{proof}

The next example elaborates that the class of fuzzy primals on a set $Y$ is not closed under the intersection operator in general.
\begin{example}
Let $Y=\{y,z\}$. Consider families $\P_1=\{\lambda\in\mathbb{I}^Y:0\leq \lambda(y)\leq 0.6,0\leq \lambda(z)\leq 1\}\cup\{0_Y\}$ and $\P_2=\{\eta\in\mathbb{I}^Y:0\leq \eta(y)\leq 1,0\leq \eta(z)\leq 0.7\}\cup\{0_Y\}$. Evidently, $\P_1,\P_2$ satisfy all conditions in Definition \ref{primaldefinition}, so they are fuzzy primals on $Y$. If $\mu,\nu$ are fuzzy sets such that $\mu(y)=0.5$, $\mu(z)=1$ and $\nu(y)=1$, $\nu(z)=0.6$, then $\nu\notin\P_1$ and $\mu\notin\P_2$. On the other hand, $\mu\cap\nu\in\P_1\cap\P_2$, but neither $\mu\in\P_1\cap\P_2$ nor $\nu\in\P_1\cap\P_2$. Thus, $\P_1\cap\P_2$ is not a fuzzy primal.
\end{example}


\section{Primal fuzzy topology}\label{sec4}

\begin{definition}
Let $(Y,\T)$ be a fuzzy topological space and $\mathcal{P}$ be a fuzzy primal on $Y$. The 3-tuple $(Y,\T,\mathcal{P})$ is said to be a primal fuzzy topological space (briefly, PFTS).
\end{definition}

\begin{definition}
Let $(Y,\T,\mathcal{P})$ be a PFTS. Then a fuzzy mapping $(\cdot)^\diamond: \mathbb{I}^Y\to \mathbb{I}^Y$ is defined as follows $\lambda^\diamond(Y,\T,\mathcal{P})=\{y_t\in\F_t(Y): \bar{\lambda}\oplus\bar{\mu}\in\mathcal{P}$ for each $\mu\in \Q(y_t)\}$ for each fuzzy subset $\lambda$. For short, we write $\lambda^\diamond$ or $\lambda_\mathcal{P}^\diamond$ instead of  $\lambda^\diamond(Y,\T,\mathcal{P})$.
\end{definition}

The next examples establish that the properties $\lambda^\diamond \subseteq \lambda$ and $\lambda \subseteq \lambda^\diamond$ are generally false.

\begin{example}\label{ex1}
Let the fuzzy set $\mu$ be defined as in Example \ref{examp}. That is,
\[
\mu(x)=\begin{cases}
	0, &\text{ if } 0\leq x\leq 1/2;\vspace{.3cm}\\
	2x-1, &\text{ if } 1/2\leq x\leq 1.
\end{cases}
\]
Consider the fuzzy topology $\T=\{0_Y,\mu,1_Y\}$ on a set $Y$ and the fuzzy primal $\P=\mathbb{I}^Y-\{1_Y\}$. 
One can check that $\mu\diamond=Cl(\mu)=1_Y\nsubseteq\mu$ (c.f., Theorem \ref{thm1} (ii)).
\end{example}

\begin{example}\label{ex2}
Consider the indiscrete fuzzy topology $\T=\{0_Y,1_Y\}$ on a set $Y$ and any fuzzy primal $\P$. Since $1_Y$ is the only fuzzy open q-neighborhood of all fuzzy points $y_t$, so for any $\lambda\in\mathbb{I}^Y$, 
\begin{align*}
y_t\in\lambda^\diamond\Longleftrightarrow\bar{\lambda}\oplus\overline{1_Y}=\bar{\lambda}\oplus 0_Y=\bar{\lambda}\in\P.\tag{$\blacksquare$}
\end{align*}

Therefore, we have
\[
\lambda^\diamond=
\begin{cases}
0_Y,  &\text{ if } \bar{\lambda}\notin\P;\vspace{.3cm}\\
1_Y,  &\text{ if } \bar{\lambda}\in\P.
\end{cases}
\]

Take $0_Y\ne\bar{\mu}\notin\P$. Thus, $\mu\nsubseteq\mu\diamond=0_Y$.
\end{example}

\begin{theorem}\label{thm1}
Let $\lambda$ and $\mu$ be fuzzy subsets of a PFTS $(Y,\T,\mathcal{P})$. Then the next statements hold true.
\begin{enumerate}[(i)]
		\item $0_Y^\diamond=0_Y$.
		\item $\lambda^\diamond=Cl(\lambda^\diamond)\subseteq Cl(\lambda)$.
		\item If $\bar{\lambda}\notin\mathcal{P}$, then $\lambda^\diamond=0_Y$.
		\item If $\lambda\subseteq \mu$, then $\lambda^\diamond\subseteq \mu^\diamond$.
		\item $(\lambda^\diamond)^\diamond\subseteq \lambda^\diamond$.
		\item  $(\lambda\cup \mu)^\diamond= \lambda^\diamond\cup \mu^\diamond$.
		\item  $(\lambda\cap \mu)^\diamond\subseteq \lambda^\diamond\cap \mu^\diamond$.
		\item If $\bar{\lambda}\notin\mathcal{P}$, then $(\lambda\cup \mu)^\diamond=\mu^\diamond$.
	\end{enumerate}
\end{theorem}
\begin{proof}
\begin{enumerate}[(i)]
	\item Since $\bar{0}_Y\oplus\bar{\lambda}=1_Y$ for any fuzzy set $\lambda$ and $1_Y\notin\mathcal{P}$, so $0_Y^\diamond=0_Y$.
	
	\item  If $y_t\in Cl(\lambda^\diamond)$, then $\lambda^\diamond\prec\nu$ for each $\nu\in\Q(y_t)$. This means that $\lambda^\diamond(z)+\nu(z)>1$ for some $z\in Y$. Let $\lambda^\diamond(z)=r$. Then $r+\nu(z)>1$ and so $2-(r+\nu)(z)<1$. This gives that $z_r\in\lambda^\diamond$ and $\nu\in\Q(z_r)$. Since $\nu\in\Q(y_t)$, then $\bar{\lambda}\oplus\bar{\nu}\in\mathcal{P}$ and so $y_t\in\lambda^\diamond$. Hence, $Cl(\lambda^\diamond)\subseteq\lambda^\diamond$. The reverse of the inclusion is always true. Thus, $Cl(\lambda^\diamond)= \lambda^\diamond$.
	
	We now show that $\lambda^\diamond\subseteq Cl(\lambda)$. If $y_t\notin Cl(\lambda)$, then $\lambda\not\prec\nu$ for some $\nu\in\Q(y_t)$. That is, $\lambda(z)+\mu(z)\leq 1$ for all $z\in X$. This implies that $\bar{\lambda}\oplus\bar{\nu}=1_Y\notin\mathcal{P}$ and hence, $y_t\notin \lambda^\diamond$. Therefore, $\lambda^\diamond\subseteq Cl(\lambda)$.
	
	\item Suppose otherwise that there exists $y_t\in\F_t(Y)$ such that $y_t\in\lambda^\diamond$. Then $\bar{\lambda}\oplus\bar{\nu}\in\mathcal{P}$ for each $\nu\in\Q(y_t)$. But, since $\bar{\lambda}\notin\mathcal{P}$, by Theorem \ref{pp1}, $\bar{\lambda}\oplus\bar{\nu}\notin\mathcal{P}$, a contradiction. Thus, $\lambda^\diamond=0_Y$.
	
	\item  Assume that $\lambda\subseteq\mu$. If $y_t\in \lambda^\diamond$, then  $\bar{\lambda}\oplus\bar{\nu}\in\mathcal{P}$ for all $\nu\in\Q(y_t)$. Since $\lambda\subseteq\mu$, so $\bar{\mu}\subseteq\bar{\lambda}$, and then, by Remark \ref{op-remark}, $\bar{\mu}\oplus\bar{\nu}\subseteq\bar{\lambda}\oplus\bar{\nu}$. This implies that $\bar{\mu}\oplus\bar{\nu}\in\mathcal{P}$. Hence,  $y_t\in \mu^\diamond$ and thus, $\lambda^\diamond\subseteq \mu^\diamond$.
	
	\item By (ii), $\lambda^{\diamond\diamond}=Cl(\lambda^{\diamond\diamond})\subseteq Cl(\lambda^\diamond)=\lambda^\diamond$. 
	
	\item Since $\lambda\subseteq\lambda\cup \mu$ and $\mu\subseteq\lambda\cup \mu$, then, by (iv), $\lambda^\diamond\subseteq$ $(\lambda\cup \mu)^\diamond$ and $\mu^\diamond\subseteq(\lambda\cup \mu)^\diamond$. It follows that $\lambda^\diamond\cup\mu^\diamond\subseteq(\lambda\cup\mu)^\diamond$. For the converse of the inclusion, if $y_t\notin\lambda^\diamond\cup\mu^\diamond$, then $y_t\notin\lambda^\diamond$ and $y_t\notin\mu^\diamond$. This implies that there exist $\nu,\eta\in\Q(y_t)$ such that $\bar{\lambda}\oplus\bar{\nu}\notin\mathcal{P}$ and $\bar{\mu}\oplus\bar{\eta}\notin\mathcal{P}$. Set $\theta=\nu\cap\eta$. Then $\theta\in\Q(y_t)$ for which $\bar{\lambda}\oplus\bar{\theta}\notin\mathcal{P}$ and $\bar{\mu}\oplus\bar{\theta}\notin\mathcal{P}$, and therefore, $(\bar{\lambda}\oplus\bar{\theta})\cap(\bar{\mu}\oplus\bar{\theta})\notin\mathcal{P}$(from Theorem \ref{pp1}). Since $\mathcal{P}$ is a fuzzy primal and $(\bar{\lambda}\oplus\bar{\theta})\cap(\bar{\mu}\oplus\bar{\theta})\subseteq(\bar{\lambda}\cap\bar{\mu})\oplus\bar{\theta}$ (from Remark \ref{op-remark}),  we get that $\overline{(\lambda\cup\mu)}\oplus\bar{\theta}=(\bar{\lambda}\cap\bar{\mu})\oplus\bar{\theta}\notin\mathcal{P}$. Thus, $y_t\notin(\lambda\cup\mu)^\diamond$. Consequently, $(\lambda\cup \mu)^\diamond\subseteq\lambda^\diamond\cup \mu^\diamond$. Hence, $(\lambda\cup \mu)^\diamond=\lambda^\diamond\cup \mu^\diamond$.
	
	\item Since $\lambda\cap\mu\subseteq\lambda$ and $\lambda\cap\mu\subseteq\mu$, then, by (iv), $(\lambda\cap\mu)^\diamond\subseteq\lambda^\diamond$ and $(\lambda\cap \mu)^\diamond\subseteq\mu^\diamond$. Therefore, 	$(\lambda\cap\mu)^\diamond\subseteq \lambda^\diamond\cap\mu^\diamond$.
	
	\item It follows from parts (iii) and (vi).
\end{enumerate}
\end{proof}

\begin{theorem}\label{closure=diamond}
Let $(Y,\T,\mathcal{P})$ be a PFTS. Then $\mu^\diamond=Cl(\mu-\bar{\lambda})$ for some $\lambda\notin\P$.
\end{theorem}
\begin{proof}
On the first hand, we want to show that $\mu^\diamond\subseteq Cl(\mu-\bar{\lambda})$ for some $\lambda\notin\P$. Let $y_t\notin Cl(\mu-\bar{\lambda})$. Then there exists $\nu\in\Q(y_t)$ such that $$\nu(y)+\big(\mu(y)-(1-\lambda(y))\big)\leq 1.$$
This implies that
$$\nu(y)+\mu(y)-1+\lambda(y)\leq 1,$$
and so
$$2-\big(\nu(y)+\mu(y)\big)\geq \lambda(y)\text{ for all }y\in Y.$$
Therefore, $\bar{\mu}\oplus\bar{\nu}\notin\P$ and thus, $y_t\notin\mu^\diamond$. This proves that $\mu^\diamond\subseteq Cl(\mu-\bar{\lambda})$.

On the other hand, if $y_t\notin\mu^\diamond$, then there exists $\theta\in\Q(y_t)$ such that $\bar{\mu}\oplus\bar{\theta}=\lambda\notin\P$. We want to check that $$\theta+\big(\mu-\bar{\lambda}\big)\leq 1_Y.$$
Given $y\in Y$. If $\theta(y)+\mu(y)>1$, then $(1-\theta)(y)+(1-\mu)(y)=\lambda(y)$ and therefore, $1-\theta(y)+1-\mu(y)=\lambda(y)\cdots(\blacktriangle)$. Thus, $\theta(y)+\mu(y)-1+\lambda(y)=1$. If $\theta(y)+\mu(y)\leq 1$, by $(\blacktriangle)$, we get $\lambda(y)=0$. This means that $\theta(y)+\mu(y)-1+\lambda(y)\leq 1$. Both cases imply that $\theta(y)+\mu(y)-1+\lambda(y)\leq 1$ means $\theta+\big(\mu-\bar{\lambda}\big)\leq 1_Y$. This yields that  $y_t\notin Cl(\mu-\bar{\lambda})$ and then  $Cl(\mu-\bar{\lambda})\subseteq\mu^\diamond$. Hence, $\mu^\diamond=Cl(\mu-\bar{\lambda})$.
\end{proof}

\begin{theorem}\label{twoprimal}
Let $\P_1,\P_2$ be fuzzy primals on a fuzzy topological space $(Y,\T)$ and let $\lambda\in\mathbb{I}^Y$. The next statements hold:
\begin{enumerate}[(i)]
	\item $\P_1\subseteq\P_2$ implies $\lambda^\diamond_{\P_1}\subseteq\lambda^\diamond_{\P_2}$.
	\item $\lambda^\diamond_{\P_1\cup\P_2}=\lambda^\diamond_{\P_1}\cup\lambda^\diamond_{\P_2}$.
\end{enumerate}
\end{theorem}
\begin{proof}
\begin{enumerate}[(i)]
	\item Let $y_t\in\lambda^\diamond_{\P_1}$. Then $\bar{\lambda}\oplus\bar{\nu}\in\P_1$ for all $\nu\in\Q(y_t)$. Since $\P_1\subseteq\P_2$, so $\bar{\lambda}\oplus\bar{\nu}\in\P_2$ for all $\nu\in\Q(y_t)$. Thus, $y_t\in\lambda^\diamond_{\P_2}$.
	
	\item  Since $\P_1\subseteq\P_1\cup\P_2$ and $\P_2\subseteq\P_1\cup\P_2$, by part (i), $\lambda^\diamond_{\P_1}\subseteq\lambda^\diamond_{\P_1\cup\P_2}$ and $\lambda^\diamond_{\P_2}\subseteq\lambda^\diamond_{\P_1\cup\P_2}$. Therefore, $\lambda^\diamond_{\P_1}\cup\lambda^\diamond_{\P_2}\subseteq\lambda^\diamond_{\P_1\cup\P_2}$.
	
	On the other hand, if $y_t\notin\lambda^\diamond_{\P_1}\cup\lambda^\diamond_{\P_2}$, then $y_t\notin\lambda^\diamond_{\P_1}$ and $y_t\notin\lambda^\diamond_{\P_2}$. This implies that there exist $\mu,\nu\in\Q(y_t)$ such that $\bar{\lambda}\oplus\bar{\mu}\notin\P_1$ and $\bar{\lambda}\oplus\bar{\nu}\notin\P_2$, respectively. Since $\mu\cap\nu\in\Q(y_t)$ and $\bar{\mu},\bar{\nu}\subseteq\bar{\mu}\cup\bar{\nu}$, by Theorem \ref{pp1}, $\bar{\lambda}\oplus(\bar{\mu}\cup\bar{\nu})\notin\P_1$ and $\bar{\lambda}\oplus(\bar{\mu}\cup\bar{\nu})\notin\P_2$. Therefore, $\bar{\lambda}\oplus\overline{\mu\cap\nu}=\bar{\lambda}\oplus(\bar{\mu}\cup\bar{\nu})\notin\P_1\cup\P_2$. Hence, $y_t\notin\lambda^\diamond_{\P_1\cup\P_2}$ and thus, $\lambda^\diamond_{\P_1\cup\P_2}\subseteq\lambda^\diamond_{\P_1}\cup\lambda^\diamond_{\P_2}$. The claim follows.
\end{enumerate}
\end{proof}


\begin{definition}
	Let $(Y,\T,\mathcal{P})$ be a PFTS. Then a fuzzy mapping $Cl^\diamond: \mathbb{I}^Y\to \mathbb{I}^Y$ is defined as follows $Cl^\diamond(\mu)=\mu\cup \mu^\diamond$ for any fuzzy set $\mu$. We may use $Cl^\diamond_\P$ instead of $Cl^\diamond$ if any kind of confusion arises. 
\end{definition}

\begin{theorem}\label{thm2}
	Let $\lambda$ and $\mu$ be fuzzy subsets of a PFTS $(Y,\T,\mathcal{P})$. Then the next statements hold true:
	\begin{enumerate}[(i)]
		\item $Cl^\diamond(0_Y)= 0_Y$.
		\item $\lambda\subseteq Cl^\diamond(\lambda)$.
		\item $\lambda\subseteq \mu$ implies $Cl^\diamond(\lambda)\subseteq Cl^\diamond(\mu)$.
		\item $Cl^\diamond(\lambda\cup \mu)= Cl^\diamond(\lambda)\cup Cl^\diamond(\mu)$.
		\item $Cl^\diamond(Cl^\diamond(\lambda))= Cl^\diamond(\lambda)$.
	\end{enumerate}
\end{theorem}
\begin{proof}
\begin{enumerate}[(i)]
	\item Since $\overline{0_Y}\notin\P$, by Theorem \ref{thm1} (iii), $0_Y^\diamond=0_Y$, and hence $Cl^\diamond(0_Y)=0_Y\cup 0_Y^\diamond=0_Y$.
	
\item This one is easy as $\lambda\subseteq\lambda\cup\lambda^\diamond=Cl^\diamond(\lambda)$.
	
\item Suppose $\lambda,\mu\in \mathbb{I}^Y$ with $\lambda\subseteq\mu$. By Theorem \ref{thm1} (iv), $\lambda^\diamond\subseteq$ $\mu^\diamond$ and so $\lambda\cup\lambda^\diamond\subseteq\mu\cup\mu^\diamond$. Thus, $Cl^\diamond(\lambda)\subseteq Cl^\diamond(\mu)$.
	
\item By the same technique used in (iii) and applying Theorem \ref{thm1} (vii), one can easily conclude that $Cl^\diamond(\lambda\cup \mu)= Cl^\diamond(\lambda)\cup Cl^\diamond(\mu)$.
	
\item To show that $Cl^\diamond(Cl^\diamond(\lambda))\subseteq Cl^\diamond(\lambda)$, we implicitly use multiple statements of Theorem \ref{thm1}. Now,
\begin{eqnarray*}
		Cl^\diamond(Cl^\diamond(\lambda))&=&Cl^\diamond(\lambda)\bigcup(Cl^\diamond(\lambda))^\diamond\\
		&=&Cl^\diamond(\lambda)\bigcup(\lambda\cup\lambda^\diamond)^\diamond\\
		&=&Cl^\diamond(\lambda)\bigcup\lambda^\diamond\bigcup(\lambda^\diamond)^\diamond\\
		&\subseteq&Cl^\diamond(\lambda)\bigcup\lambda^\diamond\bigcup\lambda^\diamond\\
		&=&Cl^\diamond(\lambda).
\end{eqnarray*}
	The converse is always true by using (ii) and (iii). Therefore, $Cl^\diamond(Cl^\diamond(\lambda))=Cl^\diamond(\lambda)$.
\end{enumerate}
\end{proof}

\begin{theorem}
Let $(Y,\T,\mathcal{P})$ be a PFTS. Then a fuzzy mapping $Cl^\diamond: \mathbb{I}^Y\to \mathbb{I}^Y$ given by $Cl^\diamond(\mu)=\mu\cup \mu^\diamond$ for any fuzzy subset $\mu$ is a Kuratowski's fuzzy closure operator.
\end{theorem}
\begin{proof}
Theorem \ref{thm2} guarantees that $Cl^\diamond$ satisfies all the postulates in Definition \ref{Kuratowski}. Thus, $Cl^\diamond$ is a Kuratowski's fuzzy closure operator.
\end{proof}

\begin{theorem}\label{primaltopology}
	Let $(Y,\T,\mathcal{P})$ be a PFTS. Then the family $\T^\diamond=\{\mu\subseteq \mathbb{I}^Y: Cl^\diamond(\bar{\mu})=\bar{\mu}\}$ constitutes a fuzzy topology on $Y$.
\end{theorem}
\begin{proof}
Since $Cl^\diamond(0_Y)=0_Y$ by Theorem \ref{thm2} (i), so $1_Y\in\T^\diamond$. By Theorem \ref{thm2} (ii), $1_Y\subseteq Cl^\diamond(1_Y)$. Then $0_Y\in\T^\diamond$.

Let $\mu,\nu\in\T^\diamond$. Then $Cl^\diamond(\bar{\mu})=\bar{\mu}$ and $Cl^\diamond(\bar{\nu})=\bar{\nu}$. By Theorem \ref{thm2} (iv), $Cl^\diamond(\overline{\mu\cap\nu})=Cl^\diamond(\bar{\mu}\cup\bar{\nu})=Cl^\diamond(\bar{\mu})\cup Cl^\diamond(\bar{\nu})=\bar{\mu}\cup\bar{\nu}=\overline{\mu\cap\nu}$. This implies that  $\mu\cap\nu\in\T^\diamond$. 

Let $\{\mu_s:s\in S\}\subseteq\T^\diamond$. Then $Cl^\diamond(\overline{\mu_s})=\overline{\mu_s}$ for all $s$. Set $\sigma=\bigcap_{s\in S}\overline{\mu_s}$. By Theorem \ref{thm2} (iii), $Cl^\diamond(\sigma)\subseteq Cl^\diamond(\overline{\mu_s})=\overline{\mu_s}$ for all $s$. This implies that $Cl^\diamond(\sigma)\subseteq \bigcap_{s\in S}\overline{\mu_s}=\sigma$. By Theorem \ref{thm2} (ii), $\sigma\subseteq Cl^\diamond(\sigma)$. Therefore, $Cl^\diamond(\bigcap_{s\in S}\overline{\mu_s})=\bigcap_{s\in S}\overline{\mu_s}$. But $\overline{\bigcap_{s\in S}\overline{\mu_s}}=\bigcup_{s\in S}\mu_s$ implies $\bigcup_{s\in S}\mu_s\in\T^\diamond$. Thus, $\T^\diamond$ is a fuzzy topology on $Y$.
\end{proof}

\begin{definition}
	We call a fuzzy topology $\T^\diamond$ produced by the above theorem the primal fuzzy topology. If it is necessary, we write $\T^\diamond_{\mathcal{P}}$ instead of $\T^\diamond$. 
\end{definition}

\begin{theorem}\label{base}
Let $(Y,\T,\mathcal{P})$ be a PFTS. Then the family $$\mathcal{B}_{\mathcal{P}}=\{\mu-\bar{\lambda}: \mu\in \T\text{ and }\lambda\not\in\mathcal{P}\}$$ is a fuzzy base for the primal fuzzy topology $\T^\diamond$ on $Y$.
\end{theorem}
\begin{proof}
We first need to check that $\mathcal{B}_{\mathcal{P}}\subseteq\T^\diamond$. If $\beta\in\mathcal{B}_{\mathcal{P}}$, then $\beta=\mu-\bar{\lambda}$ for some $\mu\in\T$ and $\lambda\not\in\mathcal{P}$. We need to show that $Cl^\diamond(\bar{\beta})=\bar{\beta}$. It is enough to prove that $(\bar{\beta})^\diamond=\bar{\beta}$. Suppose there exists $y_t\in\F_t(Y)$ such that $y_t\in(\bar{\beta})^\diamond$ but $y_t\notin(\bar{\beta})$. The first statement implies that for each $\omega\in\Q(y_t)$, we have $\overline{(\bar{\beta})}\oplus\bar{\omega}\in\P$ and so $(1_Y-(1_Y-\beta))+(1_Y-\omega)\in\P$. Therefore, 
\begin{equation}
1_Y+\beta-\omega\in\P. \tag{$\clubsuit$}
\end{equation} 
That is, 
\begin{align*}
1_Y+\beta-\omega&=1_Y+(\mu-(1_Y-\lambda))-\omega\\
&=1_Y+\mu-1_Y+\lambda-\omega\\
&=\mu-\lambda-\omega\in\P.\tag{$\spadesuit$}
\end{align*} 
On the other hand, from $y_t\notin(\bar{\beta})$, we obtain that
\begin{align*}
 t&>1-\beta(y)\\
 &=1-(\mu-(1-\lambda))(y)\\
 &=2-\mu(y)-\lambda(y).
\end{align*}
Therefore, $t+\mu(y)>2-\lambda(y)\geq 1$. This proves that 
\begin{align*}
\mu\in\Q_\T(y_t).\tag{$\bigstar$}
\end{align*}

Applying $(\spadesuit)$ to $\mu$, we have
$\lambda\leq\mu+\lambda-\mu\in\P$ implies $\lambda\in\P$, a contradiction. So $(\bar{\beta})^\diamond=\bar{\beta}$. Thus, $\mathcal{B}_{\mathcal{P}}\subseteq\T^\diamond$.

Let $y_t\in\F_t(Y)$ and let $\mu\in\Q_{\T^\diamond}(y_t)$. Then there exists $\beta\in\T^\diamond$ such that $y_t\prec\beta\subseteq\mu$. Since $\beta\in\T^\diamond$, so $\bar{\beta}$ is fuzzy $\T^\diamond$-closed and so $Cl^\diamond(\bar{\beta})=\bar{\beta}$. This means that $(\bar{\beta})^\diamond\subseteq\bar{\beta}$. Therefore, $y_t\notin(\bar{\beta})^\diamond$. This leads that there exists $\nu\in\Q_\T(y_t)\cdots(\blacklozenge)$ such that $\beta\oplus\bar{\nu}\notin\P$. By ($\clubsuit$), we have $1-\beta+\nu\notin\P$. 

Now, we need to check that $y_t\prec(\nu-\overline{(1_Y-\beta+\nu)})$. Since $y_t\prec\beta$ and $y_t\prec\nu$, then $t+\beta(y)>1$ and  $t+\nu(y)>1$. Considering, 
if $\mu(y)\leq\beta(y)$, we obtain that
\begin{align*}
t+\nu(y)-\overline{(1-\beta(y)+\nu(y))}&=t+\nu(y)-(1-(1-\beta(y)+\nu(y)))\\
&=t+\nu(y)-(-\beta(y)+\nu(y))\\
&=t+\beta(y)>1.
\end{align*}
If $\mu(y)>\beta(y)$, we also obtain that $t+\nu(y)-\overline{(1-\beta(y)+\nu(y))}=t+\beta(y)>1$. Since $\nu\in\Q_\T(y_t)$ by ($\blacklozenge$), and we have shown that for each $y_t\in\F_t(y)$ and for each $\mu\in\Q_{\T^\diamond}(y_t)$, there $(\nu-\overline{(1_Y-\beta+\nu)})\in\mathcal{B}_\P$ such that $y_t\prec(\nu-\overline{(1_Y-\beta+\nu)})\subseteq\mu$. By Definition \ref{fuzzybase}, $\mathcal{B}_\P$ is a fuzzy base for $\T^\diamond$.
\end{proof}

\begin{theorem}\label{finer}
Let $(Y,\T,\mathcal{P})$ be a PFTS. Then a primal fuzzy topology $\T^\diamond$ is finer than a fuzzy topology $\T$.
\end{theorem}
\begin{proof}
It follows from Theorem \ref{base} that $\T\subseteq\mathcal{B}_\P\subseteq\T^\diamond$. Thus, $\T\subseteq\T^\diamond$.
\end{proof}

The following illustrations demonstrate that primal fuzzy topologies are natural examples of fuzzy topologies on a universal set.

\begin{example}\label{exam-I=Finite-included}
Let $(Y, \T)$ be a fuzzy topological space, where $\T=\{0_Y,1_Y\}$, and let $y_{t_0}\in\F_t(Y)$. Suppose the fuzzy primal $\P=\{\lambda\in\mathbb{I}^Y, y_{t_0}\notin\lambda\}$. By the construction given in Example \ref{ex2}, if $\bar{\mu}\notin\P$, then $\mu^\diamond=0_Y$ and so $Cl^\diamond(\mu)=\mu$. If $\bar{\mu}\in\P$, then $\mu^\diamond=1_Y$ and so $Cl^\diamond(1_Y)=1_Y$. Therefore, each fuzzy set excluding $y_{t_0}$ is $\T^\diamond$-closed together with $1_Y$. Hence, $\T^\diamond=\{\lambda\in\mathbb{I}^Y, y_{t_0}\in\lambda\}\cup\{0_Y\}$ (This is called an included fuzzy point topology).
\end{example}

\begin{example}
	For any fuzzy topological space $(Y,\T)$, if the fuzzy primal $\mathcal{P}=\mathbb{I}^Y\setminus\{1_Y\}$, then $\T= \T^\diamond$. Let $\beta$ be any basic $\T^\diamond$-open. Then $\beta=\mu-\bar{\lambda}$ for some $\mu\in\T$ and $\lambda\notin\P$. Therefore, we have $1_Y\notin\P$ and hence, $\beta=\mu-\overline{1_Y}=\mu-0_Y=\mu\in\T$. Thus, $\T^\diamond\subseteq\T$ and $\T^\diamond\supseteq\T$ (by Theorem \ref{finer}). Hence, $\T^\diamond=\T$.
\end{example}

\begin{theorem}
Let $(Y,\T,\mathcal{P})$ be a PFTS. The following statements hold:
\begin{enumerate}[(i)]
	\item If $\bar{\lambda}\notin\P$, then $\lambda$ is fuzzy $\T^\diamond$-closed in $Y$.
	\item If $\lambda\in\mathbb{I}^Y$, then $\lambda^\diamond$ is fuzzy $\T^\diamond$-closed in $Y$.
\end{enumerate}
\end{theorem}
\begin{proof}
\begin{enumerate}[(i)]
	\item Let $\bar{\lambda}\notin\P$. By Theorem \ref{thm1} (iii), $\lambda^\diamond=0_Y$. Therefore, $Cl^\diamond(\lambda)=\lambda\cup\lambda^\diamond=\lambda$ and hence $\lambda$ is a fuzzy $\T^\diamond$-closed set in $Y$.
	
	\item Let $\lambda\in\mathbb{I}^Y$. From Theorem \ref{thm1} (v), we obtain that $Cl^\diamond(\lambda^\diamond)=(\lambda^\diamond)^\diamond\cup\lambda^\diamond=\lambda^\diamond$ leads $\lambda^\diamond$ is a fuzzy $\T^\diamond$-closed set in $Y$.
\end{enumerate}
\end{proof}

\begin{theorem}\label{zz}
Let $(Y,\T,\mathcal{P})$ and $(Y,\T,\mathcal{G})$ be two PFTSs such that $\mathcal{P}\subseteq\mathcal{G}$. Then $\T^\diamond_{\mathcal{G}}\subseteq \T^\diamond_{\mathcal{P}}$.
\end{theorem}
\begin{proof}
If $\mu\in\T^\diamond_{\mathcal{G}}$, then  $Cl^\diamond_\mathcal{G}(\bar{\mu})=\bar{\mu}\cup(\bar{\mu})_\mathcal{G}^\diamond$ implies $(\bar{\mu})_\mathcal{G}^\diamond\subseteq\bar{\mu}$. Since $\mathcal{P}\subseteq\mathcal{G}$, by Theorem \ref{twoprimal}, $(\bar{\mu})_\mathcal{P}^\diamond\subseteq(\bar{\mu})_\mathcal{G}^\diamond\subseteq\bar{\mu}$. Therefore, $Cl^\diamond_\mathcal{P}(\bar{\mu})=\bar{\mu}$ and so $\bar{\mu}$ is a fuzzy $\T^\diamond_\P$-closed set. Hence, $\mu\in\T^\diamond_{\mathcal{P}}$ and thus, $\T^\diamond_{\mathcal{G}}\subseteq \T^\diamond_{\mathcal{P}}$.
\end{proof}

\begin{theorem}
Let $(Y,\T,\mathcal{P})$ and $(Y,\T,\mathcal{G})$ be two PFTSs. Then $\ \T^\diamond_{\P\cup\mathcal{G}}= \T^\diamond_{\mathcal{P}}\bigcap\T^\diamond_{\mathcal{G}}$.
\end{theorem}
\begin{proof}
Since $\P\subseteq\P\cup\mathcal{G}$ and $\mathcal{G}\subseteq\P\cup\mathcal{G}$, by Theorem \ref{zz}, $\T^\diamond_{\P\cup\mathcal{G}}\subseteq\T^\diamond_{\P}$ and $\T^\diamond_{\P\cup\mathcal{G}}\subseteq\T^\diamond_{\mathcal{G}}$. Hence, $\T^\diamond_{\P\cup\mathcal{G}}\subseteq \T^\diamond_{\mathcal{P}}\bigcap\T^\diamond_{\mathcal{G}}$.

Conversely, if $\mu\in\T^\diamond_{\mathcal{P}}\bigcap\T^\diamond_{\mathcal{G}}$. Then $\mu\in\T^\diamond_{\mathcal{G}}$ and $\mu\in\T^\diamond_{\mathcal{G}}$. Therefore, $Cl^\diamond_\P(\bar{\mu})=\bar{\mu}$, $Cl^\diamond_\mathcal{G}(\bar{\mu})=\bar{\mu}$ implies $(\bar{\mu})^\diamond_\mathcal{P}\subseteq\bar{\mu}$ and $(\bar{\mu})^\diamond_\mathcal{G}\subseteq\bar{\mu}$. So $(\bar{\mu})^\diamond_\mathcal{P}\bigcup(\bar{\mu})^\diamond_\mathcal{G}\subseteq\bar{\mu}$. By Theorem \ref{twoprimal} (ii), $(\bar{\mu})^\diamond_{\P\cup\mathcal{G}}\subseteq\bar{\mu}$ implies $\bar{\mu}$ is fuzzy $\T^\diamond_{\P\cup\mathcal{G}}$-closed. Hence, $\mu\in\T^\diamond_{\P\cup\mathcal{G}}$ and consequently, $\T^\diamond_{\mathcal{P}}\bigcap\T^\diamond_{\mathcal{G}}\subseteq\T^\diamond_{\P\cup\mathcal{G}}$. Thus, $\T^\diamond_{\P\cup\mathcal{G}}= \T^\diamond_{\mathcal{P}}\bigcap\T^\diamond_{\mathcal{G}}$. 
\end{proof}

The following example demonstrates that $\mathcal{B}_\P$ in Theorem \ref{base} is not a fuzzy topology in general.
\begin{example}
Let $(Y,\tau)$ be an ordinary topological space. The family $\omega(\tau)$ of lower semicontinuous functions from $Y$ to $\mathbb{I}$ equipped with the standard topology is a fuzzy topology on $Y$ (see \cite{lowen1976fuzzy}). The fuzzy topology $\omega(\tau)$ is called the fuzzy topology generated by $\tau$. The base $\mathcal{B}_0$ of the standard topology includes of $(a,b)$ with $0\leq a<b\leq 1$, $[0,b)$ with $0<a\leq 1$, and $(a,1]$ with $0\leq a<1$. Let $A=\{\frac{1}{n}:n\in\mathbb{N},n\geq 2\}$. If $\P=\{\lambda:\lambda\text{ (finite)}\subseteq A\}$, by Lowen's technique (page 623-624, \cite{lowen1976fuzzy}), $\omega(\P)$ is a fuzzy primal on $Y$. Consequently, $\mathcal{B}_\P=\{\omega(\beta-\bar{\lambda}):\beta\in\mathcal{B}_0,\lambda\notin\P\}$ is a fuzzy base for $\omega(\tau)$. If $\lambda_n=\{\frac{1}{n+1},\frac{1}{n+2},\cdots\}$, then $\lambda_n\notin\P$ for all $n$. Therefore, each fuzzy set of form $\omega(\beta-\bar{\lambda}_n)\in\mathcal{B}_\P$, but $\bigcup_{n=2}^\infty\omega(\beta-\bar{\lambda}_n)\notin\mathcal{B}_\P$. 
\end{example}

The next section examines a natural condition between $\T$ and $\P$ that ensures $\mathcal{B}_\P$ is a fuzzy topology.
\section{Compatibility of a fuzzy primal with a fuzzy topology}
\begin{definition}
Let $(Y,\T,\mathcal{P})$ be a PFTS. Then $\T$ is called compatible with $\P$, denoted by $\T\approx\P$, if  for each $\lambda\in\mathbb{I}^Y$ and each $y_t\in\F_t(Y)$ with $y_t\in\lambda$, there exists $\mu\in\Q(y_t)$ such that $\bar{\lambda}\oplus\bar{\mu}\notin\P$ implies $\bar{\lambda}\notin\P$.
\end{definition}

\begin{definition}\label{square}
Let $\lambda\in\mathbb{I}^Y$. The fuzzy set $\lambda^\square$ is defined by $$\lambda^\square=\{y_t\in\F_t(Y): y_t\in\lambda \text{ and }y_t\notin\lambda^\diamond\}.$$
\end{definition}

Evidently, $\lambda^\diamond\bigcap\lambda^\square=0_Y$.

We can express any fuzzy set $\lambda$ in terms of the operators $\lambda^\diamond$ and $\lambda^\square$ using the result that follows.
\begin{theorem}\label{c1}
For any $\lambda\in\mathbb{I}^Y$, $\lambda=\lambda^\square\cup(\lambda\cap\lambda^\diamond)$.
\end{theorem}
\begin{proof}
Surely, by Definition \ref{square}, $\lambda^\square\subseteq\lambda$, and $\lambda\cap\lambda^\diamond\subseteq\lambda$. Therefore, $\lambda^\square\cup(\lambda\cap\lambda^\diamond)\subseteq\lambda$.

On the other hand, suppose $y_t\in\lambda$. If $y_t\in\lambda^\diamond$, then $y_t\in\lambda^\diamond\cap\lambda$. Differently, $y_t\in\lambda^\square$. Both cases imply $\lambda\subseteq\lambda^\square\cup(\lambda\cap\lambda^\diamond)$. Thus, $\lambda=\lambda^\square\cup(\lambda\cap\lambda^\diamond)$.
\end{proof}

\begin{theorem}\label{c2}
For any $\lambda\in\mathbb{I}^Y$, $\lambda^\square\cap(\lambda^\square)^\diamond=0_Y$.
\end{theorem}
\begin{proof}
Suppose otherwise that there is $y_t\in\F_t(Y)$ such that $y_t\in\lambda^\square\cap(\lambda^\square)^\diamond$. Then $y_t\in\lambda^\square$ and $y_t\in(\lambda^\square)^\diamond$. From the fist statement, we have $y_t\in\lambda$ and $y_t\notin\lambda^\diamond$. Form $y_t\in(\lambda^\square)^\diamond$, we have $\overline{\lambda^\square}\oplus\bar{\mu}\in\P$ for each $\mu\in\Q(y_t)$. Since $\lambda^\square\subseteq\lambda$, so $\bar{\lambda}\subseteq\overline{\lambda^\square}$ implies $\bar{\lambda}\oplus\bar{\mu}\in\P$ for each $\mu\in\Q(y_t)$. This means that $y_t\in\lambda^\diamond$, a contradiction to the consequence of the first statement. Thus, we must have $\lambda^\square\cap(\lambda^\square)^\diamond=0_Y$.
\end{proof}

\begin{theorem}\label{c3}
For any PFTS $(Y,\T,\mathcal{P})$, the following properties are equivalent:
\begin{enumerate}[(i)]
\item $\T\approx\P$.
\item For any $\lambda\in\mathbb{I}^Y$ with $\lambda\cap\lambda^\diamond=0_Y$ implies $\bar{\lambda}\notin\P$.
\item For any $\lambda\in\mathbb{I}^Y$, $\overline{(\lambda^\square)}\notin\P$.
\item For any $\lambda\in\mathbb{I}^Y$, if $\lambda$ includes no $\sigma$ such that $0_Y\neq\sigma\subseteq\sigma^\diamond$ implies $\bar{\lambda}\notin\P$.
\item For each fuzzy $\T^\diamond$-closed set $\lambda$ in $Y$, $\overline{(\lambda^\square)}\notin\P$.
\end{enumerate}
\end{theorem}
\begin{proof}
(i)$\Rightarrow$(ii) Given $\lambda\in\mathbb{I}^Y$ for which $\lambda\cap\lambda^\diamond=0_Y$. Therefore, if $y_t\in\lambda$, $y_t\notin\lambda^\diamond$ implies there exists $\nu\in\Q(y_t)$ such that $\bar{\lambda}\oplus\bar{\nu}\notin\P$. Since $\T\approx\P$, then $\bar{\lambda}\notin\P$.\\
	
(ii)$\Rightarrow$(iii) If $y_t\in\lambda^\square$, then $y_t\in\lambda$ but $y_t\notin\lambda^\diamond$. Since $\lambda^\square\subseteq\lambda$, by Theorem \ref{thm1} (iv), $y_t\notin(\lambda^\square)^\diamond$. We claim that $(\lambda^\square)^\diamond=0$. Suppose, if possible, $(\lambda^\square)^\diamond(y)=r$ for some $r\in(0,1]$, then $r<t$ and so $y_r\in(\lambda^\square)^\diamond$. Since $(\lambda^\square)^\diamond\subseteq\lambda^\diamond$, therefore, $y_r\in\lambda^\diamond$. But, since $r<t$ and $y_t\in\lambda^\square$, then $y_r\in\lambda^\square$ and hence, $y_r\notin\lambda^\diamond$, a contradiction. Thus, $(\lambda^\square)^\diamond(y)=0$ for each $y$. This implies that $\lambda^\square\cap(\lambda^\square)^\diamond=0_Y$ for each $\lambda\in\mathbb{I}^Y$. By (ii), $\overline{\lambda^\square}\notin\P$.\\

(iii)$\Rightarrow$(iv) Let $\lambda$ be a fuzzy set that possesses the property in (iv). By Theorem \ref{c1}, $\lambda=\lambda^\square\cup(\lambda\cap\lambda^\diamond)$. By Theorem \ref{thm1} (vi), $\lambda^\diamond=(\lambda^\square)^\diamond\cup(\lambda\cap\lambda^\diamond)^\diamond$. Since $\overline{\lambda^\square}\notin\P$ (by (iii)), by Theorem \ref{thm1}, $(\lambda^\square)^\diamond=0_Y$. Therefore, $\lambda^\diamond=(\lambda\cap\lambda^\diamond)^\diamond$. Since $\lambda\cap\lambda^\diamond\subseteq\lambda^\diamond$, so $\lambda\cap\lambda^\diamond\subseteq(\lambda\cap\lambda^\diamond)^\diamond$. By assumption, $(\lambda\cap\lambda^\diamond)^\diamond=0_Y$ implies $\lambda\cap\lambda^\diamond=0_Y$. This shows that $\lambda=\lambda^\square$ and thus, $\bar{\lambda}=\overline{\lambda^\square}\notin\P$.\\

(iv)$\Rightarrow$(v) Let $\lambda\in(\T^\diamond)^c$. We know that $\lambda^\square=\{y_t\in\F_t(Y): y_t\in\lambda \text{ and }y_t\notin\lambda^\diamond\}$. Since $\lambda\in(\T^\diamond)^c$, so $\lambda^\diamond\subseteq\lambda$. If $\lambda$ does not contain any non-zero fuzzy set $\sigma$ with $\sigma\subseteq\sigma^\diamond$ implies that $\lambda\cap\lambda^\diamond=0_Y$. Thus, $\lambda=\lambda^\square$ and so, by (iv), $\overline{\lambda^\square}=\bar{\lambda}\notin\P$.\\

(v)$\Rightarrow$(i) Let $\lambda\in\mathbb{I}^Y$ such that for each $y_t\in\lambda$, there exists $\nu\in\Q(y_t)$ such that $\bar{\lambda}\oplus\bar{\nu}\notin\P$. Therefore, $y_t\notin\lambda^\diamond$. Set $\mu=\lambda\cup\lambda^\diamond$. Then $\mu^\diamond=(\lambda\cup\lambda^\diamond)^\diamond=\lambda^\diamond\cup(\lambda^\diamond)^\diamond=\lambda^\diamond$ (by Theorem \ref{thm1} (v)). Thus, $Cl^\diamond(\mu)=\mu\cup\mu^\diamond=(\lambda\cup\lambda^\diamond)\cup\lambda^\diamond=\mu$. This shows that $\mu\in(\T^\diamond)^c$. By (v), $\overline{\mu^\square}\notin\P\ \cdots(\blacktriangleleft)$. If $y_t\in\mu^\square$, then $y_t\in\mu$ but $y_t\notin\mu^\diamond=\lambda^\diamond$. Since $\mu=\lambda\cup\lambda^\diamond$, so $y_t\in\lambda$. Thus, $\mu^\square\subseteq\lambda$. By assumption, we have $y_t\in\lambda$ but $y_t\notin\lambda^\diamond=\mu^\diamond$. This implies that $y_t\in\mu$ but $y_t\notin\mu^\diamond$. Therefore, $y_t\in\mu^\square$ and hence, $\mu^\square=\lambda$. By ($\blacktriangleleft$), $\overline\lambda=\overline{\mu^\square}\notin\P$. This proves that $\T\approx\P$.
\end{proof}

\begin{theorem}\label{c4}
For any PFTS $(Y,\T,\mathcal{P})$, the following properties are equivalent and implied by $\T\approx\P$.
\begin{enumerate}[(i)]
\item For any $\lambda\in\mathbb{I}^Y$ with $\lambda\cap\lambda^\diamond=0_Y$ implies $\lambda^\diamond=0_Y$.

\item For any $\lambda\in\mathbb{I}^Y$, $(\lambda^\square)^\diamond=0_Y$.

\item For any $\lambda\in\mathbb{I}^Y$, $(\lambda\cap\lambda^\diamond)^\diamond=\lambda^\diamond$.
\end{enumerate}
\end{theorem}
\begin{proof}
It follows from Theorems \ref{thm1}, \ref{c1}, and \ref{c3}.
\end{proof}

\begin{theorem}\label{c5}
Let $(Y,\T,\mathcal{P})$ be a PFTS. If $\T\approx\P$, then $(\lambda^\diamond)^\diamond=\lambda^\diamond$ for each $\lambda\in\mathbb{I}^Y$.
\end{theorem}
\begin{proof}
Given $\lambda\in\mathbb{I}^Y$. The part $(\lambda^\diamond)^\diamond\subseteq\lambda^\diamond$ follows from Theorem \ref{thm1} (v).

On the other hand, by Theorem \ref{c4} (iii), $\lambda^\diamond=(\lambda\cap\lambda^\diamond)^\diamond$. By Theorem \ref{thm1} (vii), $\lambda^\diamond=(\lambda\cap\lambda^\diamond)^\diamond\subseteq\lambda^\diamond\cap(\lambda^\diamond)^\diamond\subseteq(\lambda^\diamond)^\diamond$ and so $\lambda^\diamond\subseteq(\lambda^\diamond)^\diamond$. Thus, $\lambda^\diamond=(\lambda^\diamond)^\diamond$.
\end{proof}

\begin{theorem}\label{closed=union}
Let $(Y,\T,\mathcal{P})$ be a PFTS such that $\T\approx\P$. A set $\lambda\in\mathbb{I}^Y$ is fuzzy $\T^\diamond$-closed iff it can be written as a union of a $\T$-closed set and a fuzzy set whose complement is not in $\P$. 
\end{theorem}
\begin{proof}
Let $\lambda\in\mathbb{I}^Y$. If $\lambda$ is fuzzy $\T^\diamond$-closed, then $\lambda^\diamond\subseteq\lambda$. By Theorem \ref{c1}, $\lambda=\lambda^\square\cup\lambda^\diamond$. By Theorem \ref{thm1} (ii), $\lambda^\diamond$ is $\T$-closed in $Y$. Since $\T\approx\P$, by Theorem \ref{c3} (iii), $\overline{\lambda^\square}\notin\P$.

On the other hand, suppose that $\lambda=\mu\cup\eta$ for some $\mu\in\T^c$ and a fuzzy set $\eta$ such that $\bar{\eta}\notin\P$. It suffices to show that $\lambda^\diamond\subseteq\lambda$. Now, $\lambda^\diamond=(\mu\cup\eta)^\diamond=\mu^\diamond\cup\eta^\diamond$. Since $\bar{\eta}\notin\P$, by Theorem \ref{thm1} (iii), $\eta^\diamond=0_Y$. Therefore, by fuzzy closedness of $\mu$ and Theorem \ref{thm1} (ii), we have $\lambda^\diamond=\mu^\diamond\subseteq Cl(\mu)=\mu\subseteq\lambda$. Hence, $\lambda$ is fuzzy $\T^\diamond$-closed.
\end{proof}

\begin{theorem}\label{base=top}
Let $(Y,\T,\mathcal{P})$ be a PFTS such that $\T\approx\P$. Then $\T^\diamond=\big\{\mu-\bar{\lambda}: \mu\in \T\text{ and }\lambda\not\in\mathcal{P}\big\}$ 
\end{theorem}
\begin{proof}
Let $\mathcal{B}=\{\mu-\bar{\lambda}: \mu\in \T\text{ and }\lambda\not\in\mathcal{P}\}$. Theorem \ref{base} guarantees that $\mathcal{B}\subseteq\T^\diamond$.

On the other hand, if $\mu\in(\T^\diamond)^c$, then $\mu^\diamond\subseteq\mu$. By Theorem \ref{c1}, $\mu=\mu^\square\cup\mu^\diamond$. By Theorem \ref{closed=union}, $\mu^\diamond$ is fuzzy $\T$-closed and $\overline{\mu^\square}\notin\P$. Since $\mu^\square\cap\mu^\diamond=0_Y$, so $\mu=\mu^\square+\mu^\diamond$ and therefore, $\bar{\mu}=\overline{\mu^\square+\mu^\diamond}=(1-\mu^\diamond)-\mu^\square=\overline{\mu^\diamond}-\bar{\lambda}$, where $\overline{\mu^\diamond}\in\T$ and $\lambda=\overline{\mu^\square}\notin\P$. This implies that $\bar{\mu}\in\mathcal{B}$ and thus, $\T^\diamond\subseteq\mathcal{B}$. Consequently, $\T^\diamond=\mathcal{B}$.
\end{proof}

\section*{Conclusion and possible lines of future work}\
Chang [1] proposed a fuzzy topology on a universe, which is a modified version of classical topology. Lowen [2] expanded the definition of fuzzy topology by including all constant functions in addition to zero and unit functions.  These topological generalizations have developed into an intriguing field of study. Several methods for generating fuzzy topologies have been discussed in the literature. By investigating the concept of primal fuzzy topology, we have made a novel contribution to the field of fuzzy topology. This study is based on fuzzy primal, a supplementary concept of a fuzzy grill. Fuzzy primals can be thought of as a widening of fuzzy ideals. We have covered multiple fundamental operations on fuzzy primals. A primal fuzzy topological space combines a fuzzy topological space and a fuzzy primal. Then we have proposed a fuzzy operator, denoted by the symbol $(\cdot)^\diamond$, in regard to a fuzzy topological space. The fuzzy topological operator $(\cdot)^\diamond$ is used to define another operator called "$Cl^\diamond$." The characteristics of $Cl^\diamond$ have been discussed in detail. We have observed that $Cl^\diamond$ is consistent with all of Kuratowski's fuzzy closure operator's axioms, therefore it generates a fuzzy topology known as primal fuzzy topology. Some examples have been provided to demonstrate the fact that primal fuzzy topologies are natural (non-trivial) fuzzy topologies. We have shown that the base of primal fuzzy consists of fuzzy sets that are fuzzy open sets from the original fuzzy topology minus the complement of some element in a fuzzy primal. Consequently, a primal fuzzy topology can be viewed as a larger fuzzy topology than the original one. We defined the term as a fuzzy primal being compatible with a fuzzy topology and examined other analogous conditions. It is demonstrated that the fuzzy base which produces the primary fuzzy topology is also a fuzzy topology if a fuzzy primal and a fuzzy topology are compatible.

The findings in this article are preliminary, and additional studies could provide additional knowledge by investigating more aspects of primal fuzzy topology, such as primal fuzzy interior, primal fuzzy limit points, and so on. Separation axioms, compactness, and connectedness of primal fuzzy topologies are also possible research directions in this area.

\medskip
\noindent {\bf Funding:} This research has received no external funding.\\
{\bf Conflicts of interest:} The authors declare no conflicts of interest.\\
{\bf Availability of data and material:} No data were used to support this study. \\
{\bf Code availability:} Not applicable.\\


\end{document}